\documentclass[12pt]{article}
\usepackage{amsmath, amsfonts, graphicx, amsthm, color, enumerate, hyperref, amssymb}

\DeclareMathOperator{\id}{id}

\DeclareMathOperator{\Mat}{Mat}

\DeclareMathOperator{\G}{G}

\DeclareMathOperator{\q}{q}

\DeclareMathOperator{\spa}{span}

\DeclareMathOperator{\GL}{GL}

\DeclareMathOperator{\Sp}{Sp}

\DeclareMathOperator{\diag}{Diag}
\DeclareMathOperator{\Orth}{O}

\DeclareMathOperator{\rad}{rad}

\newtheorem{thm}[subsection]{Theorem}
\newtheorem*{thm*}{Theorem}
\newtheorem{prop}[subsection]{Proposition}
\newtheorem*{prop*}{Proposition}
\newtheorem{lemma}[subsection]{Lemma}
\newtheorem*{lemma*}{Lemma}
\newtheorem{corollary}[subsection]{Corollary}
\newtheorem*{corollary*}{Corollary}

\newtheorem{definition}[subsection]{Definition}

\title{On involutions of type $\Orth(\q,k)$ over a field of characteristic two}
\author{Mark Hunnell \\ \textit{Winston-Salem State University} \\ hunnellm@wssu.edu \\ \\
John Hutchens \\ \textit{Winston-Salem State University} \\ hutchensjd@wssu.edu \\  \\ Nathaniel Schwartz \\ \textit{Baltimore, MD}}

\begin{document}
\maketitle

\vfill
\noindent Declarations of interest: none \\
This research did not receive any specific grant from funding agencies in the public, commercial, or not-for-profit sectors.
\newpage

\section{Introduction}

In this article we study the involutions of orthogonal groups over fields of characteristic $2$.  Throughout the paper $k$ denotes a field.  An understanding of these involutions is beneficial to furthering the study of symmetric $k$-varieties, a generalization of symmetric spaces, to fields of characteristic 2.  Symmetric spaces were first studied by Gantmacher in \cite{ga39} in order to classify simple real Lie groups.  In \cite{be57} Berger provides a complete classification of symmetric spaces for simple real Lie algebras.   The primary motivation is to extend Helminck's \cite{he00} study of $k$-involutions and symmetric $k$-varieties to include fields of characterstic $2$.  This has been studied for groups of type $\G_2$ and $\mathrm{A}_n$ in \cite{hs18, sc18} and over fields of characteristic not $2$ in \cite{do06, bhw15, he02, bdhw16, hu14, hu15,hu16}.  We also extend the results of Aschbacher and Seitz \cite{as76} who studied similar structures for finite fields of characteristic $2$.  

The study of involutions gives us a way to describe generalized symmetric spaces or symmetric $k$-varieties of the form $G(k)/H(k)$ where $G(k)$ is an algebraic group over $k$ and $H(k)$ is the fixed point group of some automorphism of order $2$ on $G(k)$.  The notation for the theory of algebraic groups is standard and introduced as needed.  We use Hoffman and Laghribi's \cite{hl04} almost exclusively for notation concerning quadratic forms over fields of characteristic $2$. 

There have been many studies of orthogonal groups over fields of characteristic $2$.  In \cite{ch66} Cheng Hao discusses automorphisms of the orthogonal group over perfect fields of characteristic $2$ when the quadratic form is nondefective.  Pollak discusses orthogonal groups over global fields of characteristic $2$ in the case the quadratic form is nondefective in \cite{po70} and Connors writes about automorphism groups of orthogonal groups over fields of characteristic $2$ in \cite{co73,co74,co75,co76} for a nondegenerate quadratic form.   We extend these results by including discussions of defective and degenerate quadratic forms.

We also extend the results of Wiitala from \cite{wi78}.  The following result appears as Theorem A in \cite{wi78}, where
\[ \tau_{u}(w) = w + \dfrac{B(u,w)}{\q(u)} u. \]

\begin{thm}
\label{wiitala}
Let $\q$ be a quadratic form on a vector space $V$ over a field $k$ of characteristic $2$ such that $\rad(V)$ is empty with respect to $\q$.  If $\tau \in \Orth(\q,k)$, then $\tau$ is an involution if and only if $\tau = \tau_{l} \cdots \tau_{2}\tau_{1}$ and
\begin{enumerate}
\item $\tau_i=\tau_{u_i}$ is a transvection with respect to $u_i$ for all $i$, or 
\item each $\tau_{i}$ is an involution with respect to a hyperbolic space.
\end{enumerate}
\end{thm}

The author goes on to note that all such involutions of the same type and length are $\GL(V)$-conjugate.   These results are extended in this article to a vector space with nontrivial radical and the study of conjugacy classes under $\Orth(\q,k)$.

Our main results appear in section 3 and concern the characterization of conjugacy classes of involutions in a maximal nonsingular subspace and a characterization of what we call radical involutions.  We go on to discuss the general case and some special cases within.  We prove a characterization of $\Orth(\q,k)$-conjugacy for three types of involutions.  First in Theorem \ref{conj_transv} we show that two diagonal involutions $\tau_{u_l} \cdots \tau_{u_2} \tau_{u_1}$ and  $\tau_{x_l} \cdots \tau_{x_2} \tau_{x_1}$ are $\Orth(\q,k)$-conjugate if and only if a bilinear form induced by the norms of $u_1, u_2, \ldots$ and $u_l$ is equivalent to the bilinear form induced by the norms of $x_1, x_2, \ldots$ and $x_l$.  See Section $3$ for a more precise statement. Proposition \ref{conj_null} deals with involutions with respect to a hyperbolic space, which are also known as null involutions.  We show that two null involutions are $\Orth(\q,k)$-conjugate if and only if they have the same number of reduced factors.  Finally, radical involutions are described in Corollary \ref{conj_rad}, which establishes that all radical involutions satisfying a certain norm condition are conjugate.  The paper concludes with a discussion of the involutions in the case that $V$ is singular, but not totally singular.

\section{Preliminaries}
The following definitions can be found in \cite{hl04}.  Let $k$ be a field of characteristic 2 and $V$ a vector space defined over $k$.  We call $\q:V \rightarrow k$ a \emph{quadratic form} if it satisfies $\q(av) = a^2 \q(v)$ for all $a \in k$, $v \in V$and there exists a symmetric bilinear form $B: V \times V \rightarrow k$ such that $\q(w+w') = \q(w)+\q(w') + B(w,w')$ for all $w,w' \in V$.  Over fields of characteristic $2$ nonsingular symmetric bilinear forms are also symplectic.

The pair $(V,\q)$ is called a \emph{quadratic space}.
Given a quadratic form, there exists a basis of $V$, consisting of $e_i, f_i, g_j$, where $i \in \{1, 2, \hdots, r\}$ and $j \in \{1, 2, \hdots, s\}$ and field elements $a_i, b_i, c_j \in k$ such that 
\[\q(w) = \displaystyle \sum_{i=1}^r (a_i x_i^2 + x_iy_i + b_iy_i^2) + \sum_{j=1}^s c_j z_j^2\]
when $w = \sum_{i=1}^r (x_ie_i + y_if_i) + \sum_{j=1}^s z_j g_j$.  We denote this quadratic form by
\[\q = [a_1,b_1] \perp [a_2,b_2] \perp \cdots \perp [a_r,b_r]\perp \langle c_1, c_2, \hdots, c_s  \rangle \]
where $\rad(V)= \spa \{ g_1, g_2, \hdots, g_s\} $ is the \emph{radical} of $V$.  We say that such a quadratic form is of type $(r,s)$.  A nonzero vector $v \in V$ is an \emph{isotropic vector} if $\q(v)=0$, $V$ is an \emph{isotropic vector space} if it contains isotropic elements and \emph{anisotropic} otherwise.  The vector space $V$ is called \emph{nonsingular} if $s=0$, and is called \emph{nondefective} if $s=0$ or $\rad(V)$ is anisotropic.  A \emph{hyperbolic plane} has a quadratic form isometric to the form $[0,0]$ and will be denoted by $\mathbb{H}$.  We will call $\q'$ a \emph{subform} of $\q$ if there exists a form $\mathrm{p}$ such that $\q \cong \q' \perp \mathrm{p}$.  

Suppose $\mathcal{P} $ is a totally singular subspace of $V$ with basis $\{p_1, p_2, \dots, p_l\}$, then for $w=\sum_{i=1}^l w_ip_i$, $w^{\prime}=\sum_{i=1}^lw^{\prime}_ip_i$, and field elements $a_i \in k$, we will denote the diagonal bilinear form
\[ B(w,w') = a_1w_1w'_1 + a_2w_2w'_2 + \cdots + a_lw_lw'_l, \]
by $\langle a_1, a_2, \cdots, a_l \rangle_B$, following \cite{hl04}.   

We will denote $\mathbb{H}\perp \mathbb{H} \perp \dots \perp \mathbb{H}$, where there are $i$ copies of $\mathbb{H}$ in the decomposition, by $i \times \mathbb{H}$.  Similarly, $\langle 0, 0, \dots, 0 \rangle$, where the $0$ is repeated $j$ times, will be denoted $j\times \langle 0 \rangle$.
The following is Proposition 2.4 from \cite{hl04}. 

\begin{prop}\label{hofflaghdecomp}
Let $\q$ be a quadratic form over $k$.  There are integers $i$ and $j$ such that
\[ \q \cong i \times \mathbb{H} \perp \widetilde{\q_r} \perp \widetilde{\q_s} \perp j \times \langle 0 \rangle, \]
with $\widetilde{q_r}$ nonsingular, $\widetilde{\q_s}$ totally singular and $\widetilde{\q_r} \perp \widetilde{q_s}$ anisotropic.  The form $\widetilde{\q_r} \perp \widetilde{\q_s}$ is uniquely determined up to isometry.  In particular $i$ and $j$ are uniquely determined.  
\end{prop}

We call $i$ the \emph{Witt index} and $j$ the \emph{defect} of $\q$.  If 
\[ \q \cong i \times \mathbb{H} \perp \widetilde{\q_r} \perp  j \times \langle 0 \rangle \perp \widetilde{\q_s}, \]
with respect to the basis 
\[ \{ e_1,f_1, \ldots e_i,f_i, \ldots , e_r,f_r, g_1, \ldots g_j, g_{j+1}, \ldots , g_{s} \}, \]
we will call 
\[ \mathrm{def}(V) = \mathrm{span}\{g_1,\ldots g_j \}, \]
the \emph{defect} of $V$.  

If $\mathcal{W}$ is a basis for a subspace $W$ of $V$, we will refer to the restriction of $\q$ to $W$ by $\q \vert_W$ or sometimes $\q_W$.

If $G$ is an algebraic group, then an automorphism $\theta: G \rightarrow G$ is an \emph{involution} if $\theta^2 = \id$, $\theta \neq \id$.  In addition, $\theta$ is a $k$-\emph{involution} if $\theta(G(k)) = G(k)$, where $G(k)$ denotes the $k$-rational points of $G$.  We define the {\em fixed point group} of $\theta$ in $G(k)$ by
\[ G(k)^{\theta} = \{ \gamma \in G(k) \ | \ \theta \gamma \theta^{-1} = \gamma \}. \]
This is often denoted $H(k)$ or $H_k$ in the literature when there is no ambiguity with respect to $\theta$.  Notice that since $\theta$ has order $2$, this group is also the centralizer of $\theta$ in $G(k)$.  We will use $k^*$ to denote the nonzero elements of $k$ and $k^2$ to denote the subfield of $k$ that consists of the squares of $k$.  When $k$ is a perfect field we have $k=k^2$.  An $l$-tuple of elements of the set $S$ will be denoted by $S^{\times l}$.

We often consider groups that leave a bilinear form or a quadratic form invariant.  If $B$ is a bilinear form on a nonsingular vector space $V$ we will denote the {\em symplectic group} of $(V,\q)$ by
\[ \Sp(B,k) = \{ \varphi \in \GL(V) \ | \ B(\varphi(w), \varphi(w')) = B(w,w') \text{ for } w,w' \in V \}. \]
The classification of involutions for $\Sp(B,k)$ for a field $k$ such that $\mathrm{char}(k) \neq 2$ has been studied in \cite{bhw15}.  For any quadratic space $V$ we will denote the {\em orthogonal group} of $(V,\q)$ by
\[ \Orth(\q,k) = \{  \varphi \in \GL(V) \ | \ \q(\varphi(w)) = \q(w) \text{ for } w \in V \}. \]
When $V$ is nonsingluar we have $\Orth(\q,k) \subset \Sp(B,k)$ if $B$ is the bilinear form that is associated with $\q$,
\[ B(w,w') = \q(w+w') + \q(w) + \q(w'). \]
We define $\textrm{BL}(B,k) = \{ \varphi \in \GL(V) \ | \ B(\varphi(w), \varphi(w') ) = B(w,w') \}$.  Notice that when $V$ is nonsingular $\textrm{BL}(B,k) \cong \Sp(B,k)$, and in general $\textrm{BL}(B,k) \supset \Orth(\q,k)$.  We have the isomorphism
\[ \mathrm{BL}(B,k) \cong (\Sp(B_{V_{\mathcal{B}}},k) \times \GL(\rad(V)))\ltimes \mathrm{Mat}_{2r,s}(k), \]
where $\dim_k(V_{\mathcal{B}}) = 2r$ and $V= V_{\mathcal{B}} \perp \rad(V)$. 

We will need to make use of some simple facts about quadratic spaces stated in the following lemmas.  The first outlines some standard isometries for quadratic forms over a field of characteristic $2$, and the second allows us to express $\q$ using a different completion of the nonsingular space.  These and more like them appear in \cite{hl04}.  

\begin{lemma}
Let $\q$ be a quadratic form on a vector space $V$, and suppose $\alpha \in k$.  Then the following are equivalent representations of $q$ on $V$:
\begin{enumerate}
\item $[a,b] = [a,a+b+1] = [b,a] = [ \alpha^2 a, \alpha^{-2} b] $
\item $[a,b] \perp [c,d] = [a+c,b] \perp [c,b+d] = [c,d] \perp [a,b]$
 \end{enumerate}
\end{lemma}

\begin{lemma} Let $c_i, c_i^{\prime}, d_i \in k$ for $1 \leq i \leq n$, and denote the subfield of squares in $k$ by $k^2$.  Suppose $\{c_1, ... , c_n \}$ and $\{c_1^{\prime}, ... , c_n^{\prime} \}$ span the same vector space over $k^2$ and $\q = [c_1,d_1] \perp \hdots \perp [c_n , d_n]$.  Then there exist $d_i^{\prime} \in k$, $1 \leq i \leq n$, such that $\q = [c_1^{\prime},d_1^{\prime}] \perp \hdots \perp [c_n^{\prime} , d_n^{\prime}]$.
\end{lemma}

\section{Nonsingular Involutions}

Now we study the isomorphism classes of involutions of $\Orth(\q,k)$ when $(V,\q)$ is nonsingular. Recall that in general $\Sp(B,k) \supset \Orth(\q,k)$ when $B$ is induced by $\q$ on $V$ and $V$ is nonsingular.  A \emph{symplectic transvection} with respect to $u \in V$ and $a \in k$ is  a map of the form
\[ \tau_{u,a}(w) = w + aB(u,w) u, \]
and such a map is an \emph{orthogonal transvection} if $\q(u) \neq 0$  and $a = \q(u)^{-1}$.  Notice that for a symplectic transvection $a$ is allowed to be zero, but $\tau_{u,0} = \id$.  The symplectic group is generated by symplectic transvections and the orthogonal group is generated by orthogonal transvections as long as $V$ is not of the from $V=\mathbb{H} \perp \mathbb{H}$ over $\mathbb{F}_2$ as pointed out in Theorem 14.16 in \cite{gr02}.  A \emph{symplectic involution} is a map of order $2$ in $\Sp(B,k)$.   

An involution $\sigma \in \Sp(B,k)$ is called \emph{hyperbolic} if $B(v,\sigma(v)) = 0$ for all $v \in V$, and \emph{diagonal} otherwise.  Observe that all nontrivial hyperbolic elements of $\Sp(B,k)$ are involutions.

If $\sigma \in \Sp(B,k)$, then we call $R_{\sigma}=(\sigma - \id_V)V$ the residual space of $\sigma$ and define $\mathrm{res } (\sigma) = \dim R_{\sigma}$.  Then the following comes from \cite{om78}:
\begin{thm}
\label{omeara}
Let $\sigma \in \Sp(B,k), \ \sigma^2=\id_V, \ \sigma \neq \id_V$.  Then: 
\begin{enumerate}
	\item If $\sigma$ is hyperbolic, then $\sigma$ is a product of $\mathrm{res} ( \sigma) + 1$, but not of $\mathrm{res} (\sigma)$, symplectic transvections. 
	\item If $\sigma$ is diagonal, then $\sigma$ is a product of $\mathrm{res} ( \sigma)$, but not of $\mathrm{res} (\sigma) -1$, symplectic transvections. 
	\item In either case, the vectors inducing transvections whose composition is $\sigma$ are mutually orthogonal. 
\end{enumerate}
\end{thm}

Consider the symplectic involution of the form
\[ \tau_{u_l,a_l} \cdots \tau_{u_2,a_2} \tau_{u_1,a_1}. \]
If $a =[a_i] \in k^{\times l}$ and $\mathcal{U} = \{u_1,u_2,\ldots, u_l\}$, then we use $\tau_{\mathcal{U},a}$ to denote this map.
We may assume $\mathcal{U}$ consists of mutually orthogonal vectors in $V$, thus  $\spa \mathcal{U}$ is a singular subspace of $V$ with dimension  less than or equal to $l$.  A factorization of a transvection involution is called \emph{reduced} if it is written using the least number of factors, and the number of factors in a reduced expression is called the \emph{length} of the involution.  

\begin{lemma}
If $\sigma \in \Sp(B,k)$ is diagonal and we let $r=\mathrm{res} ( \sigma)$, then there exists a set $\mathcal{U}=\{ u_1, u_2, \dots, u_r \} $, where $B(u_i,u_j)=0$ for all $\{i,j\} \subset [l]$, and $a = [a_i] \in (k^*)^{\times r}$ such that $\mathcal{U}$ is a basis for $R_{\sigma}$ and $\sigma = \tau_{\mathcal{U},a}$.
\end{lemma}
\begin{proof}
By \ref{omeara}, we know $\sigma$ is a product of $r$ transvections whose inducing vectors are mutually orthogonal.  $R_{\sigma}$ is the span of these vectors, and $r=\dim (R_{\sigma})$, therefore these vectors must be linearly independent.
\end{proof}

We want to know when two diagonal involutions of the same length are equal, and to that end we define the following relationship.  Consider a pairing consisting of a list of $l$ orthogonal vectors contained in a nonsingular vector space over a field of characteristic $2$ along with a vector in $(k^*)^{\times l}$, where $k^*$ denotes the nonzero elements of $k$.  This vector is our ordered list of $a_i$'s and we take the components in $k^*$, since we can assume we have a reduced diagonal involution of length $l$.  Let $\mathcal{U}$ be as above and let 
\[ \mathcal{X} = \{x_1,x_2,\ldots, x_l\}, \]
$a = (a_1,a_2, \ldots, a_l)$ and $b=(b_1,b_2,\ldots, b_l)$.  The pairing $(\mathcal{U},a)$ and $(\mathcal{X},b)$ is called \emph{involution compatible} if $\mathcal{U}$ and $\mathcal{X}$ span the same $l$-dimensional singular subspace of $V$ such that  $u_i = \sum \alpha_{ji} x_j$ and the following hold
\begin{align}
b_j &= \sum a_i \alpha_{ji}^2 \text{ and} \\ 
0 &= \sum a_i \alpha_{ji} \alpha _{ki} \text{ for all } \{j,k\} \subset [l].
\end{align}
Notice that this is equivalent to 
\[ [ \alpha_{ij} ]_{1\leq i,j \leq l}^T \diag(a_1,\ldots, a_l) [\alpha_{ij}]_{1\leq i,j \leq l} =  \diag(b_1,\ldots, b_l). \]
We can simplify the statement by setting $A=[\alpha_{ij}]_{1\leq i,j \leq l}$ and $\diag(a_1,\ldots, a_l) = [a_i]$
\begin{equation} \label{symp_inv_mat}
A^T[a_i] A = [b_i], 
\end{equation}
and we can see this is equivalent to 
\[ \langle a_1,a_2,\ldots, a_l \rangle_B \cong \langle b_1, b_2, \ldots, b_l \rangle_B, \]
an equivalence of bilinear forms.

\begin{thm}
Let $\tau_{\mathcal{U},a}$ and $\tau_{\mathcal{X},b}$ be diagonal involutions.
Then $\tau_{\mathcal{U},a}=\tau_{\mathcal{X},b}$ if and only if $(\mathcal{U},a)$ and $(\mathcal{X},b)$ are involution compatible.
\end{thm}
\begin{proof}
Suppose $\tau_{\mathcal{U},a} = \tau_{\mathcal{X},b}$.  Then for $w \in V$,
\begin{equation}\label{sympinveq}
a_1B(u_1,w)u_1 +  \cdots + a_lB(u_l,w)u_l = b_1 B(x_1,w) x_1 + \cdots b_l B(x_l,w)x_l. 
\end{equation}
For each $u_i$ there exists a $v_i$ such that the set of $v_i$ provide a nonsingular completion of dimension $2l$.   Choosing $w=v_i$ we see that
\[ a_i u_i = b_1 B(x_1,v_i) x_1 + \cdots b_l B(x_l,v_i) x_l. \]
This shows that $\mathcal{U}$ and $\mathcal{X}$ span the same nonsingular space.  We choose coefficients for $u_i$ in terms of $\mathcal{X}$ as
\[ u_i = \displaystyle\sum_{j=1}^l \alpha_{ji}x_j. \]
Now substituting our new expression into Equation \ref{sympinveq} and replacing $w$ with $y_j$ such that $B(x_k,y_k) = 1$ and $B(x_j, y_k) = 0$ when $j \neq k$ we have 
\begin{equation}\label{sympinveq2}
a_1B\left(\sum_{j=1}^l \alpha_{j1}x_j,y_k \right)\sum_{j=1}^l \alpha_{j1}x_j+  \cdots + a_lB\left(\sum_{j=1}^l \alpha_{jl}x_j,y_k \right)\sum_{j=1}^l \alpha_{jl}x_j = b_k x_k. 
\end{equation}
Now simplifying the bilinear forms we arrive at Equation \ref{symp_inv_mat}.

If we assume that $(\mathcal{U},a)$ and $(\mathcal{X},b)$ are involution compatible we can reconstruct Equation \ref{sympinveq} from basis vectors and we have $\tau_{\mathcal{U},a} = \tau_{\mathcal{X},b}$.
\end{proof}

\begin{corollary}
Two diagonal involutions $\tau_{\mathcal{U},a}$ and $\tau_{\mathcal{X},b}$ are $\Sp(B,k)$-conjugate if and only if there exists $\mathcal{X}'$ such that $(\mathcal{X}',a)$ is involution compatible with $(\mathcal{X},b)$.
\end{corollary}

In 2.1.8 of \cite{om78} the following Theorem is stated.

\begin{thm}
Let $\sigma \in \Sp(B,k)$ be hyperbolic with residual space $R_{\sigma}$.  Let $\tau$ be any transvection such that $R_{\tau} \subset R_{\sigma}$.  Then $R_{\tau \sigma} = R_{\sigma}$, but $\tau \sigma$ is not hyperbolic.
\end{thm}

The next result describes how hyperbolic maps relate to equivalent diagonal maps.

\begin{lemma}
Let $\sigma, \ \theta \in \Sp(B,k)$ be hyperbolic.  Then $\sigma = \theta$ if and only if there exists a symplectic transvection $\tau_{u,a} \in \Sp(B,k)$ where $u \in R_{\sigma}$ and $a \in k^*$, such that $\tau_{u,a}\sigma = \tau_{u,a}\theta$.
\end{lemma}
\begin{proof}
	If $\sigma = \theta$, then one may choose any $u \in R_{\sigma}=R_{\theta}$, $a \in k^*$.  Now if such a $\tau_{u,a}$ exists, then $\sigma = \theta$ since $\tau_{u,a}^2 = \id_V$.
\end{proof}

Let $\tau_{u,a}$ be a symplectic involution and notice that $\tau_{u,a} \in \Orth(q,k)$ only if
\begin{align*}
\q\left(\tau_{u,a}(w) \right) &= \q(w + a B(w,u)u) \\
&= \q(w) + \q\left( a B(w,u)u \right) + B\left(w,a B(w,u)u \right)\\
&= \q(w) + a^2B(w,u)^2 \q(u) + a B(w,u)^2 \\
&= \q(w),
\end{align*}
so $B(w,u)^2a(a\q(u) + 1) = 0$.  So either $B(w,u) = 0$ for all $w$, $a=0$ or $\q(u) = 1/a$.  Therefore we will refer to $\tau_{u,\frac{1}{\q(u)}}$ by $\tau_u$. 

\begin{prop} \label{two_transv}
Two orthogonal tranvections $\tau_u$ and $\tau_x$ are equal if and only if $x = \alpha u$ for some $\alpha \in k$.
\end{prop}
\begin{proof}
First assuming $x = \alpha u$, we have
\begin{align*}
\tau_{\alpha u}(w) &= w + \frac{B(\alpha u, w )}{\q(\alpha u)} \alpha u \\
&= w + \frac{\alpha B(u, w )}{\alpha^2\q( u)} \alpha u \\
&= \tau_u(w).
\end{align*}
Therefore $\tau_u = \tau_x$.

Now consider $\tau _u = \tau _x$.  Then 
\begin{align*}
w + \frac{B(u, w )}{\q(u)} u &= w + \frac{B(x,w)}{\q(x)}x \\
\frac{B(u, w )}{\q(u)} u &= \frac{B(x,w)}{\q(x)}x \\
u &= \frac{B(x,w)\q(u)}{B(u,w) \q(x)}x.  
\end{align*}
Therefore, setting $\alpha = \frac{B(x,w)\q(u)}{B(u,w)\q(x)}$, we have $u= \alpha x$.
\end{proof}

\begin{prop}
Let $\phi \in \Orth(\q,k)$.  Then for a product of transvections
\[  \tau_{u_1}\tau_{u_2}\cdots \tau_{u_l} \in \Orth(\q,k), \]
we have the conjugation relation
\[ \phi \tau_{u_1}\tau_{u_2}\cdots \tau_{u_l} \phi^{-1} = \tau_{\phi(u_1)}\tau_{\phi(u_2)} \cdots \tau_{\phi(u_l)}. \]
\label{conj_beh}
\end{prop}
\begin{proof}
First notice that
\[ \phi \tau_{u} \phi^{-1} (w) = w + \dfrac{B(u,\phi^{-1}(w))}{\q(u)} \phi(u) = w +\dfrac{B(\phi(u),w)}{\q(\phi(u))} \phi(u) = \tau_{\phi(u)}(w). \]
Now we see that
\begin{align*}
\phi \tau_{u_1}\tau_{u_2}\cdots \tau_{u_l} \phi^{-1} &= \phi \tau_{u_1} \phi^{-1} \phi \tau_{u_2} \phi^{-1}\cdots \phi \tau_{u_l} \phi^{-1} \\
&=\tau_{\phi(u_1)}\tau_{\phi(u_2)} \cdots \tau_{\phi(u_l)},
\end{align*}
as required.
\end{proof}

Consider the reduced diagonal involution 
\[ \tau = \tau_{u_1} \tau_{u_2} \cdots \tau_{u_l}, \]
where as before $\mathcal{U} = \{u_1, u_2, \ldots, u_l\}$ are mutually orthogonal vectors.  If we consider the subspace $ \spa \mathcal{U}\subset V$,
then we have
\[ \q\vert_{\text{span}\mathcal{U}} \sim \langle \q(u_1),\q(u_2), \ldots, \q(u_l) \rangle. \]

\begin{prop} \label{1_over_q_to_q}
If $\q(u_i) \neq 0$ for $1\leq i \leq l$ then
\[ \left\langle \dfrac{1}{\q(u_1)}, \dfrac{1}{\q(u_2)}, \ldots, \dfrac{1}{\q(u_l)} \right\rangle_B \cong  \left\langle \dfrac{1}{\q(x_1)}, \dfrac{1}{\q(x_2)}, \ldots, \dfrac{1}{\q(x_l)} \right\rangle_B \]
if and only if 
\[ \left\langle \q(u_1), \q(u_2), \ldots, \q(u_l) \right\rangle_B \cong  \left\langle \q(x_1), \q(x_2), \ldots, \q(x_l) \right\rangle_B. \]
\end{prop}
\begin{proof}
If
\[ \left\langle \dfrac{1}{\q(u_1)}, \dfrac{1}{\q(u_2)}, \ldots, \dfrac{1}{\q(u_l)} \right\rangle_B \cong  \left\langle \dfrac{1}{\q(x_1)}, \dfrac{1}{\q(x_2)}, \ldots, \dfrac{1}{\q(x_l)} \right\rangle_B, \]
then there exists some $A$ such that 
\[ A^T\left[\dfrac{1}{\q(u_i)} \right] A =  \left[\dfrac{1}{\q(x_i)}\right]. \]
Notice that
\[ [\q(u_i)] [\q(x_i)] A^T \left[\dfrac{1}{\q(u_i)} \right] A [\q(x_i)][\q(u_i)] = \left[\q(u_i)^2\q(x_i)\right] \]
and letting $A'=[\q(u_i)]^{-1}[\q(x_i)]^{-1}A[\q(x_i)]\q(u_i)]$ then
\[ ([\q(x_i)]A'[\q(u_i)]^{-1})^T[ \q(u_i)]([\q(x_i)]A'[\q(u_i)]^{-1}) =  \left[\q(x_i)\right]. \]
This gives us 
\[ \left\langle \dfrac{1}{\q(u_1)}, \dfrac{1}{\q(u_2)}, \ldots, \dfrac{1}{\q(u_l)} \right\rangle_B \cong  \left\langle \dfrac{1}{\q(x_1)}, \dfrac{1}{\q(x_2)}, \ldots, \dfrac{1}{\q(x_l)} \right\rangle_B \]
implies
\[ \left\langle \q(u_1), \q(u_2), \ldots, \q(u_l) \right\rangle_B \cong  \left\langle \q(x_1), \q(x_2), \ldots, \q(x_l) \right\rangle_B, \]
and the argument is reversible for the converse.
\end{proof}

\begin{corollary}\label{invcomp_norm}
If
\[ \left\langle \q(u_1), \q(u_2), \ldots, \q(u_l) \right\rangle_B \cong  \left\langle \q(x_1), \q(x_2), \ldots, \q(x_l) \right\rangle_B, \]
then 
\[ \left\langle \q(u_1), \q(u_2), \ldots, \q(u_l) \right\rangle \cong  \left\langle \q(x_1), \q(x_2), \ldots, \q(x_l) \right\rangle. \]
\end{corollary}

In general the converse of Corollary \ref{invcomp_norm} is not true.  In particular consider two diagonal involutions of length $2$ 
\[ \tau_{u_2}\tau_{u_1}, \tau_{x_2}\tau_{x_1} \in \Orth(\q,k) \]
over $k = \mathbb{F}_2(t_1,t_2)$ such that
\begin{align*}
\q(x_1) &= \q(u_1) + t_1^2\q(u_2) \text{ and} \\
\q(x_2) &= \q(u_1) + \q(u_2).
\end{align*}
Let $\q(u_1) =1$ and $\q(u_2)=t_2$.  Notice that $\q(x_1),\q(x_2) \in k^2[\q(u_1,\q(u_2)]$, which gives us that $\langle \q(u_1), \q(u_2) \rangle \cong \langle \q(x_1),\q(x_2) \rangle$.  In this case $\q(u_1)$ and $\q(u_2)$ form a basis for a $k^2$-vector space of dimension $2$ and so do $\q(x_1)$ and $\q(x_2)$.  Therefore any matrix $A$ such that $A^T[\q(u_i)]A = [\q(x_i)]$ and $A=[\alpha_{kj}]$ must have $\alpha_{11}=\alpha_{12}=\alpha_{22}=1$ and $\alpha_{21}=t_1$ and so
\[ \begin{bmatrix}
	1 & 1 \\
	t_1 & 1
	\end{bmatrix}^T 
	 \begin{bmatrix}
	1 & 0 \\
	0 & t_2
	\end{bmatrix}
	 \begin{bmatrix}
	1 & 1 \\
	t_1 & 1
	\end{bmatrix} = 
	 \begin{bmatrix}
	1 + t_1^2 t_2  & 1 + t_1t_2  \\
	 1 + t_1t_2 & 1 + t_2
	\end{bmatrix}. \]
Since that off diagonal entries, $1+t_1t_2$, are not zero the conditions for $\langle \q(u_1),\q(u_2) \rangle_B \cong \langle \q(x_1),\q(x_2) \rangle_B$ are not satisfied and we have a counter example.
\begin{lemma}\label{equal_tr}
Two orthogonal involutions given by reduced products of orthogonal transvections are equal, $\tau_{u_l} \cdots \tau_{u_2} \tau_{u_1} = \tau_{x_l} \cdots \tau_{x_2} \tau_{x_1}$, if and only if 
\[ \spa \{u_1, u_2, \ldots, u_l \} = \spa\{x_1,x_2, \ldots, x_l\}, \]
and
\[ \left\langle \q(u_1), \q(u_2), \ldots, \q(u_l) \right\rangle_B \cong  \left\langle \q(x_1), \q(x_2), \ldots, \q(x_l) \right\rangle_B. \]
\end{lemma}
\begin{proof}
Let $\{u_1, u_2, \ldots, u_l \}$ and $\{x_1,x_2, \ldots, x_l\}$ be sets of linearly independent mutually orthongonal vectors, none of which are in $\rad(V)$ and all of which have nonzero norm.  Now assume  $\tau_{u_l} \cdots \tau_{u_2} \tau_{u_1} = \tau_{x_l} \cdots \tau_{x_2} \tau_{x_1}$.  Then for each set of linearly independent vectors there exists a completion of the symplectic basis.   In particular there exists a set $\{v_1,v_2, \ldots, v_l\}$ of linearly independent vectors in $V$ such that $B(u_i, v_j) =1$ when $i=j$ and zero otherwise.  Notice that we can define $\tau_{u_i}$ by
\[ \tau_{u_i}(v_i) = v_i + \dfrac{B(u_i,v_i)}{\q(u_i)} u_i= v_i + \dfrac{1}{\q(u_i)} u_i, \]
and this transvection acts as the identity on every other basis vector.  Setting
\[  \tau_{u_l} \cdots \tau_{u_2} \tau_{u_1}(v_i) = \tau_{x_l} \cdots \tau_{x_2} \tau_{x_1} (v_i), \]
we arrive at the equation
\begin{equation}\label{equal_tr_eq}
\dfrac{1}{\q(u_i)} u_i = \dfrac{B(x_1,v_i)}{\q(x_1)} x_1 +  \dfrac{B(x_2,v_i)}{\q(x_2)} x_2 + \cdots +  \dfrac{B(x_l,v_i)}{\q(x_l)} x_l, 
\end{equation}
which tells us in particular that we can write each $u_i$ as a linear combination of $ \{x_1,x_2, \ldots, x_l\}$ and the two sets must span the same space.  Notice that since we can write each $x_j$ as a linear combination of $\{u_1, u_2, \ldots, u_l \}$ that the constants $B(x_j,v_i)=\alpha_{ij}$ are just the $i$-th component of $x_j$ written in the $u$-basis.  In other words we can write
\[ x_j = \alpha_{1j} u_1 + \alpha _{2j} u_2 + \cdots + \alpha_{lj}u_l. \]

Now let us assume that we can write each $u_i$ in the $x$-basis and set
\[ u_i = \beta_{i1}x_1 + \beta _{i2} x_2 + \cdots + \beta_{il}x_l. \]
Solving for $\beta_{ij}$ in terms of $\alpha$'s if $A = [\alpha_{ij}]_{1\leq i,j \leq l}$ we arrive at the condition
\[ A^T \left[ \dfrac{1}{\q(u_i)} \right] A = \left[ \dfrac{1}{\q(x_i)} \right]. \]
Then by Proposition \ref{1_over_q_to_q} we have the result.
\end{proof}

We will use the following result, which is Lemma 2.6 from \cite{hl04}.

\begin{lemma}\label{zero_cancel}
Let $\q$ and $\q'$ be nondefective quadratic forms of the same dimension.  If 
\[ \q \perp j \times \langle 0 \rangle \cong \q' \perp j \times \langle 0 \rangle, \]
then $\q \cong \q'$.
\end{lemma}

The following is a Gram-Schmidt type theorem for characteristic 2.

\begin{lemma}\label{gramschmidt_char2}
Let $V$ be a symplectic space of dimension $2r$.  Given $\{e_1, e_2, ... , e_r \} \subset V$, a linearly independent set of vectors such that $e_i \perp e_j$, there exists $\{e_1^{\prime}, f_1,e_2^{\prime}, f_2, ... , e_r^{\prime}, f_r\} \subset V$ such that $B(e_i^{\prime},f_j) = \delta_{ij}, and \ B(f_i,f_j) = B(e_i^{\prime},e_j^{\prime})=0$.
\end{lemma}
\begin{proof}
Choose $f_1 \in V$ such that $B(e_1,f_1) = \alpha \neq 0$.  Define $e_1^{\prime} = \frac{1}{\alpha} e_1$ and $e_i^{\prime} = e_i + \frac{B(e_i,f_1)}{\alpha} e_1$ for $i \in \{ 2, 3, ..., r \}$, so that $B(f_1, e_j^{\prime}) = \delta_{1j}$.  Then $V = \langle e_1^{\prime}, f_1 \rangle \perp V^{\prime}$, where $\dim(V^{\prime}) < \dim(V)$, and induction establishes the result.
\end{proof}

\begin{thm} \label{conj_transv}
Let $\tau_{u_l} \cdots \tau_{u_2} \tau_{u_1}$ and $\tau_{x_l} \cdots \tau_{x_2} \tau_{x_1}$ be orthogonal diagonal involutions on $V$ such that $\phi \in \Orth(\q,k)$.  Then 
\[ \phi \tau_{u_l} \cdots \tau_{u_2} \tau_{u_1}\phi^{-1} = \tau_{x_l} \cdots \tau_{x_2} \tau_{x_1} \]
if and only if 
\[ \left\langle \q(u_1), \q(u_2), \ldots, \q(u_l) \right\rangle_B \cong  \left\langle \q(x_1), \q(x_2), \ldots, \q(x_l) \right\rangle_B. \]
\end{thm}
\begin{proof}
First notice that the above condition is stronger than the two spaces having isometric norms.  Recall from Proposition \ref{conj_beh} that we have
\[ \phi \tau_{u_l} \cdots \tau_{u_2} \tau_{u_1}\phi^{-1} = \tau_{\phi(u_l)} \cdots \tau_{\phi(u_2)} \tau_{\phi(u_1)}. \]
If we assume that the two involutions are $\Orth(\q,k)$-conjugate we have
\[ \tau_{\phi(u_l)} \cdots \tau_{\phi(u_2)} \tau_{\phi(u_1)}  = \tau_{x_l} \cdots \tau_{x_2} \tau_{x_1}, \]
and so
\[ \left\langle \q(\phi(u_1)), \q(\phi(u_2)), \ldots, \q(\phi(u_l)) \right\rangle_B \cong  \left\langle \q(x_1), \q(x_2), \ldots, \q(x_l) \right\rangle_B \]
and
\[ \left\langle \q(u_1), \q(u_2), \ldots, \q(u_l) \right\rangle_B =\left\langle \q(\phi(u_1)), \q(\phi(u_2)), \ldots, \q(\phi(u_l)) \right\rangle_B . \]

Now let us assume $\langle \q(u_1), \q(u_2), \ldots, \q(u_l) \rangle_B \cong \langle \q(x_1), \q(x_2), \ldots, \q(x_l) \rangle_B$.  Then $\langle \q(u_1), \q(u_2), \ldots, \q(u_l) \rangle \cong \langle \q(x_1), \q(x_2), \ldots, \q(x_l) \rangle$ and there exists a map $\phi \in \Orth(\q,k)$ such that $\phi(\spa \mathcal{U}) = \spa \mathcal{X}$, where $\mathcal{U} = \{u_1, u_2, \ldots, u_l \} \text{ and }  \mathcal{X} =  \{x_1,x_2, \ldots, x_l\}$.  We already know that the equivalent bilinear form condition is met so by Lemma \ref{equal_tr} we have that 
\[ \tau_{\phi(u_l)} \cdots \tau_{\phi(u_2)} \tau_{\phi(u_1)}  = \tau_{x_l} \cdots \tau_{x_2} \tau_{x_1}, \]
and so the two involutions are conjugate.
\end{proof}

\subsection{Null Involutions}

In this section we discuss involutions of the second type in Theorem \ref{wiitala}.  This definition can also be found in \cite{sn89}.  We note that basic null involutions are hyperbolic, in the sense of \cite{om78}.

\begin{definition} A plane $P = \spa\{ e, f \}$ is hyperbolic (or Artinian) if both of the following are satisfied:
	\begin{enumerate}
		\item $\q(e) = \q(f) = 0$ 
		\item $B(e,f) \neq 0.$
	\end{enumerate}
\end{definition}

If $e, f$ span a hyperbolic plane, we can rescale to assume $B(e,f) =1$.  Proposition 188.2 of \cite{sn89} guarantees that every nonsingular nonzero isotropic vector is contained in a hyperbolic plane.  
\begin{definition}
Let $\eta$ be an involution of $\Orth(\q,k)$ where $(V,\q)$ is a quadratic space, and let $\mathbb{P}$ be the orthogonal sum of two hyperbolic planes.  Then $\eta$ is called a basic null involution in $\mathbb{P}$ if all of the following are satisfied:
	\begin{enumerate}
		\item $\eta$ leaves $\mathbb{P}$ invariant
		\item $\eta$ fixes a 2-dimensional subspace of $\mathbb{P}$ where every vector has norm zero
		\item $\eta \vert_{\mathbb{P}^C} = \id_{\mathbb{P}^C}$, where $\mathbb{P}^C$ is the complement of $\mathbb{P}$ in $V$.
	\end{enumerate}
\end{definition}

An automorphism $\eta$ is a basic null involution in $\mathbb{P}$ if and only if there exists a basis with respect to which the matrix of $\eta$ is a pair of $2 \times 2$ Jordan blocks, all eigenvalues are 1 and acts as the identity elsewhere.  This happens precisely when there is a basis $\{ e_1 , f_2, e_2, f_1 \}$ for $A$ such that $B(e_i ,f_i) =1$, $\eta(e_i) = e_i$, $\eta(f_1) = e_2 + f_1$, and $\eta(f_2) = e_1 + f_2$.  

\begin{prop} \label{conj_null}
Two null involutions are $\Orth(\q,k)$-conjugate if and only if they have the same length.
\end{prop}
\begin{proof}
Let $\eta_k \cdots \eta_2 \eta_1$ be a null involution on $V$ where $\eta_i$, $1 \leq i \leq k$ are basic null involutions. Then each $\eta_i$ corresponds to a four dimensional space made up of two perpendicular hyberbolic planes.  In other words for each $\eta_i$ there exists a subspace $N_i$ such that $\q\vert_{N_i} \sim [0,0]\perp[0,0]$.  Similarly, let $\mu_k \cdots \mu_2\mu_1$ be a null involution whose basic null involutions $\mu_i$ have corresponding four dimensional hyperbolic subspaces $M_i$ such that $\q\vert_{M_i} \sim [0,0]\perp[0,0]$.  If we choose $\phi \in \Orth(\q,k)$ such that $\phi: N_i \to M_i$, then
\[ \phi \mu_k \cdots \mu_2\mu_1 \phi^{-1} = \eta_k \cdots \eta_2 \eta_1. \] 
If two null involutions do not have the same length they are not $\GL(V)$-conjugate.
\end{proof}

We recall from \cite{wi78} that if a map is a product of an othogonal transvection and a basic null involution, then it is also the product of three orthogonal transvections.  

\subsection{Radical Involutions}

In this section we characterize the involutions acting in the radical of $V$.  Recall the bilinear from is identically zero here.  First let us consider the following result.

\begin{prop} \label{Oqk_radV}
$\Orth(\q|_{\rad(V)},k) \cong \GL_j(k) \ltimes \Mat_{j,s-j}(k)$ where $j$ is the defect of $\rad(V)$.
\end{prop}
\begin{proof}

By Proposition \ref{hofflaghdecomp} every norm on $\rad(V)$ is isometric to
\[ \langle 0, \cdots, 0, c_{j+1}, \cdots, c_s \rangle, \]
where $j$ is the defect of $\q$ and $\dim (\rad(V)) = s$.  Now the subform
\[ \langle c_{j+1}, \cdots, c_s \rangle, \]
is anisotropic.  We can choose a basis
\[ \mathcal{R} = \{ g_1, g_2, \ldots, g_j, g_{j+1}, g_{j+2}, \ldots, g_s \} \]
of $\textrm{rad}(V)$ such that
\[ \q(g_i) = 0 \text{ for } 1 \leq i \leq j \]
and
\[ \q(g_i) = c_i \text{ for } j+1 \leq i \leq s. \]
Let us denote the vector space spanned by the basis vectors of $\mathcal{R}$ with nonzero norms by
\[ \text{span}\{g_{j+1}, g_{j+2}, \ldots, g_s \} = \textrm{def}(V)_{\mathcal{R}}'. \]
 If $\phi \in \Orth(\q\vert_{\rad(V)},k)$  then the image of $\phi$ is defined by the four linear maps $\chi: \textrm{def}(V) \to \textrm{def}(V)$, $M: \textrm{def}(V)_{\mathcal{R}}' \to \textrm{def}(V)$, $N: \textrm{def}(V) \to \textrm{def}(V)_{\mathcal{R}}'$ and $\psi : \textrm{def}(V)_{\mathcal{R}}' \to \textrm{def}(V)_{\mathcal{R}}'$.  Let $x\in \textrm{def}(V)$ then $\phi(x) = \chi(x) + Nx$ where $Nx \in \textrm{def}(V)_{\mathcal{R}}'$.  Also $\q(\phi(x)) = \q(x) = 0$, so
\[ \q(\chi(x) + Nx) = \q(\chi(x)) + \q(Nx) = 0. \]
Now $\chi(x) \in \textrm{def}(V)$ so $\q(\chi(x) ) = 0$.  Leaving $\q(Nx) = 0$.  There are no nontrivial vectors in $\textrm{def}(V)_{\mathcal{R}}'$ such that $\q(Nx) = 0$ therefore $N=0$.  In general we require $\chi \in \textrm{GL}_j(k)$ such that $\q(\chi(x))=0$, but $\q$ is identically zero, so $\chi$ can be any element of $\textrm{GL}_j(k)$.

Now consider $y \in \textrm{def}(V)_{\mathcal{R}}'$.  If $\phi(y) = My + \psi(y)$, then 
\begin{align*}
\q(\phi(y)) &= \q(My + \psi(y)) \\
&= \q(My) + \q(\psi(y))
\end{align*}
The vector $My \in \textrm{def}(V)$ so $\q(My) = 0$ and we have $\q(y) = \q(\psi(y))$ for $\psi(y) \in \textrm{def}(V)_{\mathcal{R}}'$, but this means $\psi(y) = y$ for all $y \in \textrm{def}(V)_{\mathcal{R}}'$.  In other words $\psi = \id$.  We end up with $\phi(y) = y + My$ and since $\q(My)=0$ for any $M$, the map $M$ can be any element of $\Mat_{j,s-j}(k)$.

Consider two elements $\phi_1, \phi_2 \in \Orth(\q\vert_{\rad(V)},k)$ defined by maps $\chi_1, M_1$ and $\chi_2, M_2$ respectively.   Any element in $\rad(V)$ can be written as $x+y \in \rad(V) = \textrm{def}(V) \oplus \textrm{def}(V)_{\mathcal{R}}'$ where $x \in \textrm{def}(V)$ and $y \in \textrm{def}(V)_{\mathcal{R}}'$. We see that
\[ \phi_1\phi_2(x+y) = \chi_1\chi_2(x) + y + (M_1 + \chi_1M_2)y. \]
This is equivalent to the action of the block matrices acting on $\rad(V)$ so we have defined an isomorphism 
\[ \Psi:\Orth(\q\vert_{\rad(V)},k) \to 
\begin{bmatrix}
	\textrm{GL}_j(k) & \textrm{Mat}_{j,s-j}(k) \\
	0& \id
	\end{bmatrix}, \]
such that
\[ \Psi(\phi) = \Psi(\chi,M) =  \begin{bmatrix}
	\chi & M \\
	0& \id
	\end{bmatrix}. \]
Further, we can verify that the subgroup of the form
\[ \left\{ \begin{bmatrix}
	\id & M \\
	0& \id
	\end{bmatrix}  \ \bigg| \ M \in \textrm{Mat}_{j,s-j}(k) \right\} \]
is normal in $\left[\begin{smallmatrix}
	\textrm{GL}_j(k) & \textrm{Mat}_{j,s-j}(k) \\
	0& \id
	\end{smallmatrix} \right]$.
\end{proof}

If $\theta \in \Orth(\q,V)$ such that $\dim (\rad(V)) >1$, then to preserve the bilinear form we must have $\theta(\rad(V)) =\rad(V)$. 

We define a \emph{radical involution} to be an element $\rho \in \Orth(\q,k)$ that acts trivially outside of the $\rad(V)$ and is of order $2$.  Each nontrivial orthogonal transformation on $\rad(V)$ detects a defective vector in $V$.   For example if $\rho(g) = g'$ then $\q(g+g') = \q(g) + \q(g') = 0$.  A \emph{basic radical involution} is a map $\rho_i \in \Orth(\q,k)$ such that $\rho_i(g_i) = g_i'$ where $g_i, g_i'$ are linearly independent vectors in $\rad(V)$ with $\q(g_i) = \q(g_i')$.

\begin{prop}
Every radical involution can be written as a finite product of basic radical involutions.
\end{prop}
\begin{proof}
Let $\rho$ be a radical involution on $V$.  There is a vector $g_1 \in \rad(V)$ such that $\rho$ acts nontrivially on $g_1$.  Then there must be a vector $g_1' \in \rad(V)$ that is linearly independent form $g_1$, or else order or $\rho$ is not $2$, such that $\q(g_1') = \q(g_1)$ and $\rho(g_1) = g_1'$.  Now $\{g_1,g_1'\}$ forms a basis for a two dimensional subspace $\rad(V)_1\subset \rad(V)$ with defect $\geq 1$.  If $g_1$ and $g_1'$ are the only vectors where $\rho$ acts nontrivially then we are done and $\rho=\rho_1$ is a basic radical involution.  If not there exists an element $g_2 \in \rad(V)$ such that $g_2 \not\in \rad(V)_1$ and $\rho(g_2) = g_2'$ defines a nontrivial action.  Otherwise $g_2 \in \rad(V)_1$ and $\rho$ is already defined on $\rad(V)_1$.  So $g_2, g_2'$ are linearly independent from one another and from $\rad(V)_1$.  We define $\rad(V)_2$ to be the span of $\{g_1,g_1',g_2,g_2'\}$.  If $\rho$ acts trivially outside $\rad(V)_2$ we are done and $\rho = \rho_2\rho_1$. For any $\rad(V)_i$ either $\rho$ acts trivially outside of $\rad(V)_i$ and $\rho = \rho_i \cdots \rho_2\rho_1$ or there exists a new vector $g_{i+1}$ that is linearly independent.  By induction we have that there exists a basis 
\[ \{ g_1, g_1', g_2, g_2', \ldots, g_l,g_l', h_{l+1}, \ldots, h_s\}, \]
of $\rad(V)$ such that $\dim(\rad(V)) = s$ and $\rho(g_i) = g_i'$ for all $1\leq i \leq l$ and $\rho(h_j) = h_j$ for all $l+1 \leq j \leq s$.
Each $\rho_i$ acts nontrivially on the subspace spanned by $\{g_i, g_i'\}$ and trivially on remaining basis vectors.  So $\rho = \rho_l \cdots \rho_2\rho_1$ is a product of basis radical involutions.
\end{proof}

\begin{prop}
Two basic radical involutions $\rho_1, \rho_2$ are $\Orth(\q,k)$-conjugate if and only if $\rho_1$ and $\rho_2$ act non-trivially on isometric vectors.
\end{prop}
\begin{proof}
Let $\rho_1(g_1) = g_1'$ and $\rho_2(g_2) = g_2'$.  Then
\[ \delta \rho_1 \delta^{-1}(g_2) = \rho_2(g_2), \]
if and only if $\delta^{-1}(g_2) = g_1$. 
\end{proof}

Each radical involution maps an element $g_i \mapsto g_i'$ with $\q(g_i) = \q(g_i')$.  We chose a basis of $\rad(V)$ with respect to $\rho$ of length $m$ to be
\[ \{g_1 + g_1', g_1, g_2+g_2',g_2, \ldots, g_m+g_m', g_m, h_{2m+1}, \ldots, h_s \}, \]
where $\rho$ acts nontrivially on $g_i + g_i'$ and $h_j$.  We define the \emph{quadratic signature} of the radical involution to be 
\[  \langle \q(g_1), \q(g_2), \ldots, \q(g_m) \rangle. \]

\begin{corollary} \label{conj_rad}
All radical involutions of length $m$ with same quadratic signature
\[  \langle \q(g_1), \q(g_2), \ldots, \q(g_m) \rangle, \]
are conjugate.
\end{corollary}

\section{Involutions of a general vector space}

Elements in $\Orth(\q,k)$ where $(V,\q)$ is a quadratic space and $\dim (\rad(V) ) \geq 0$, can be thought of in terms of block matrices.  Consider a matrix of the form
\begin{equation} \label{gen_map}
(\tau, Y , \rho) =  \begin{bmatrix}
	\tau & 0 \\
	Y & \rho
	\end{bmatrix}, 
\end{equation}
where $\tau \in \Sp(B_{V_{\mathcal{B}}},k)$ and where $\mathcal{B}$ is a basis of some maximal nonsingular space in $V$ with $\dim( V_{\mathcal{B}}) = 2r$, $\dim ( \rad(V) ) = s$ and $\dim( V ) = 2r+s$.  Now we know that $\rad(V)$ must be left invariant by such a map so $\rho \in \Orth(\q_{\rad(V)}, k)$ and $(\tau,Y) \in \mathcal{M}(\q,V_{\mathcal{B}})$,  where 
\[ \mathcal{M}(\q,V_{\mathcal{B}}) = \{ (\phi, X) \in \Sp(B_{V_{\mathcal{B}}}, k) \ltimes \Mat_{2r,s} \ | \  \q(\phi(w) ) = \q(w+Xw) \}.  \]

Let $\q$ be a quadratic form of type $(r,s)$ on a vector space $V$ over a field $k$ of characterstic $2$ with $\dim (V) = 2r+s$.  Let us define 
\[ \mathcal{B} = \{ u_1,v_1,u_2,v_2, \ldots, u_r,v_r \}, \]
to be some basis of a maximal nonsingular subspace of $V$ of dimension $2r$.   Then 
\[ V = V_{\mathcal{B}} \perp \rad(V), \]
where $V_{\mathcal{B}} = \spa \mathcal{B}$.  We are interested in the case when $\dim ( \rad(V) ) = s >1$ as all elements of $\Orth(\q,k)$ leave $\rad(V)$ invariant.

\begin{prop} \label{cor_rad_inv}
$(\tau, Y, \rho)^2  = \id$ if and only if $\tau^2 = \id$ , $\rho^2=\id$ and $Y = \rho Y \tau$.
\end{prop}
\begin{proof}
Thinking of $(\tau,Y,\rho)$ as a block matrix we have
\[ \begin{bmatrix}
	\tau & 0 \\
	Y & \rho
	\end{bmatrix}^2 =
	\begin{bmatrix}
	\tau^2 & 0 \\
	Y\tau+ \rho Y & \rho^2
	\end{bmatrix}. \]
This matrix is order $2$ if and only if $\tau^2 = \id$, $\rho^2 = \id$ and $Y = \rho Y \tau$.
\end{proof}

There are two main types of maps of order $2$ of this form to consider.  First we notice that if the above map has order $2$ it is necessary that $Y \tau = \rho Y$.

\begin{prop}
If $\tau,\rho \in \Orth(\q,k)$ such that $\tau$ is a diagonal involution and $\rho$ is a radical involution, then there exists a map $Y:V_{\mathcal{B}_{\tau}} \to \rad(V)$ such that $\widetilde{Y}=\left[\begin{smallmatrix}
	\id & 0 \\
	Y & \id
	\end{smallmatrix} \right] \in \Orth(\q,k)$ and  $(\tau, Y, \rho) $ is an involution on $V$.
\end{prop}
\begin{proof}
Let $V$ be a vector space over a field of characteristic $2$ with a quadratic form $\q$ of type $(r,s)$,
\[ \tau = \tau_{u_l} \cdots \tau_{u_2} \tau_{u_1}, \]
and define 
\[ \mathcal{B}_{\tau} = \{u_1, u_1', u_2, u_2', \ldots, u_l,u_l', w_1, w_1', \ldots , w_{2(r-l)}, w_{2(r-l)}' \}, \]
so that we have the decomposition $V = V_{\mathcal{B_{\tau}}} \perp \rad(V)$ and $W$ is the subspace of $V_{\mathcal{B}_{\tau}}$ such that $\tau|_W = \id_W$. We can define 
\begin{align*}
\widetilde{Y}(u_i) &= u_i + h_i + \rho(h_i) \\
\widetilde{Y}(u_i') & = u_i' + \dfrac{1}{\q(u_i)}\left( h_i + \rho(h_i) \right) \\
\widetilde{Y}(w_j) &= w_j.
\end{align*}
Notice that $ h_i + \rho(h_i)$ is a vector in $\rad(V)$ such that $\q(h_i + \rho(h_i))=0$.  A direct computation shows that the properties in Proposition \ref{cor_rad_inv} are met and $(\tau, Y, \rho) $ is an involution in $\Orth(\q,k)$.
\end{proof}

Moreover, the above $(\tau, Y, \rho) $ is such that $u_i \mapsto u_i + (h_i + \rho(h_i))$ and so $\mathcal{B}_{\tau}$ is shifted by $h_i + \rho(h_i)$ and $\tau_{u_i} \mapsto \tau_{u_i + (h_i + \rho(h_i))}$.

\begin{prop}
A map of the form 
\[ (\phi, X, \delta) = \begin{bmatrix}
			\phi & 0 \\
			X & \delta
			\end{bmatrix}, \]
is an element of $\Orth(\q,k)$ if and only if $\phi \in \Sp(B_{V_{\mathcal{B}}},k)$, $\delta \in \Orth(\q_{\rad(V)},k)$ and $\q(\widetilde{X}(w)) = \q(\phi(w))$ for all $w \in V_{\mathcal{B}}$ and $\widetilde{X} = \id_V +X$.
\label{norm_cnd}
\end{prop}
\begin{proof}
Let $w \in V_{\mathcal{B}}$ and $g \in \rad(V)$ and assume $(\phi, X, \delta) \in \Orth(\q,k)$.  Then we have
\begin{align*}
\q(w+g) &= \q((\phi,X,\delta)(w+g)) \\
&= \q(\phi(w) + Xw + \delta(g)) \\
&= \q(\phi(w)) + \q(Xw) + \q(\delta(g)). 
\end{align*}
Recall that $\q(w+g) = \q(w) + \q(g)$ and $\q(\delta(g)) = \q(g)$ to establish
\[ \q(w) + \q(g) = \q(\phi(w))  + \q(Xw) + \q(g).  \]
This is true since $\delta \in \Orth(\q_{\rad(V)},k)$ and $(\phi, X, \delta)$ must leave $\rad(V)$ invariant.  So setting
\[ \widetilde{X} =  \begin{bmatrix}
			\id & 0 \\
			X & \id
			\end{bmatrix}, \]
we have
\[ \q(w) + \q(Xw) = \q(\phi(w)) \Rightarrow \q(\widetilde{X}(w)) = \q(\phi(w)). \]

Now assuming that $\phi \in \Sp(B_{V_{\mathcal{B}}},k)$, $\delta \in \Orth(\q_{\rad(V)},k)$ and $\q(\widetilde{X}(w)) = \q(\phi(w))$ we can reverse the argument.
\end{proof}

The property in Proposition \ref{norm_cnd} is preserved under
composition as we now note. We can consider the product
\[ (\phi, X, \delta)(\phi', X', \delta') = (\phi \phi', X \phi' + \delta X', \delta \delta'). \]
We may also compute
\begin{align*}
\q((X \phi' + \delta X')(w)) &= \q(X \phi'(w)) + \q(\delta(X'w)) \\
&= \q(\phi\phi'(w)) + \q(\phi'(w)) + \q(X'w) \\
&=  \q(\phi\phi'(w)) + \q(\phi'(w))  + \q(\phi'(w)) + \q(w) \\
&= \q(\phi\phi'(w)) + \q(w),
\end{align*}
which is equivalent.

The purpose of the next result is to establish that any map of the form $(\tau_{\mathcal{U},a}, Y, \rho) \in \Orth(\q,k)$, where $\tau_{\mathcal{U},a}$ is a symplectic involution, can be written with an orthogonal involution in the first component.

\begin{prop} \label{qYw=0}
Every involution of the form $(\tau_{\mathcal{U},a},Y, \rho)$ can be written as
\[ (\tau_{\mathcal{U}'}, Y', \rho) = (\tau_{\mathcal{U}'}, 0, \id)(\id, Y', \id)(\id, 0, \rho), \]
where each of the three maps in the decomposition is in $\Orth(\q,k)$.
\end{prop}
\begin{proof}
Assume that $a_i  \in k^*$ for all $i$ otherwise the corresponding factor would be trivial.  We can choose a basis such that $\q(Yw)=0$ for all $w \in V_{\mathcal{B}_{\tau_{\mathcal{U}'}}}$ by replacing $u_i$ with
\[ u_i' = u_i + \dfrac{1}{a_i} Yv_i. \]
To see that this works we first observe that 
\[ (\tau_{\mathcal{U}}, Y, \rho)(u_i) = u_i + Yu_i, \]
where $Yu_i \in \rad(V)$.  Then computing the norm of $u_i \in \mathcal{B}_{\tau_{\mathcal{U}}}$ we have
\begin{align*}
\q\left( (\tau_{\mathcal{U}}, Y, \rho)(u_i) \right) &= \q( u_i + Yu_i ) \\
\q(u_i) &= \q(u_i) + \q(Yu_i).
\end{align*}
Simplifying, we see that $\q(Yu_i)=0$.

There is a set vectors in the nonsingular completion of $\mathcal{U}$, which we will label $v_i$ such that $B(u_i,v_i) = 1$.  These vectors are not fixed by $\tau_{\mathcal{U}}$.  Computing the image of $v_i$ we have
\begin{align*}
(\tau_{\mathcal{U}}, Y, \rho)(v_i) &= v_i + a_i B(u_i,v_i)u_i + Yv_i \\
&= v_i + a_i u_i + Yv_i.
\end{align*}
We take the norm of the image of $v_i$
\begin{align*}
\q\left( (\tau_{\mathcal{U}}, Y, \rho)(v_i) \right) &= \q( v_i + a_i u_i + Yv_i ) \\
\q(v_i) &= \q(v_i + a_iu_i) + \q(Yv_i) \\
&= \q(v_i) + a_i^2\q(u_i) + B(v_i, a_iu_i) + \q(Yv_i) \\
&= \q(v_i) + a_i^2\q(u_i) + a_i + \q(Yv_i).
\end{align*}
We can solve for $\q(Yv_i)$ and see that
\[ \q(Yv_i) = a_i^2 \q(u_i) + a_i. \]
Notice here that $\q(Yv_i)=0$ only if $a_i=0$ or $\q(u_i) = 1/a_i$.  We have assumed $a_i \neq 0$ and if $\q(u_i) = 1/a_i$, $\tau_{u_i,a_i}$ is already an orthogonal transvection.  Let us compute the norm of $u_i' = u_i + \dfrac{1}{a_i} Yv_i$,
\begin{align*}
\q\left( u_i + \dfrac{1}{a_i} Yv_i \right) &= \q(u_i) + \dfrac{1}{a_i^2} \q(Yv_i) \\
&= \q(u_i) + \dfrac{1}{a_i^2} \left( a_i^2 \q(u_i) + a_i \right) \\
&= \q(u_i) + \q(u_i) + \dfrac{1}{a_i} \\
&= \dfrac{1}{a_i}.
\end{align*}

Now we can verify that $\tau_{u_i,a_i} = \tau_{u_i'}$ for all $i$, which is enough to say that $\tau_{\mathcal{U},a} = \tau_{\mathcal{U}'}$.  First notice that 
\[ B(u_i,u_i') = B\left(u_i, u_i + \dfrac{1}{a_i} Yv_i \right) = 0, \]
 which tells us that $\tau_{\mathcal{U}'}$ fixes $\mathcal{U}$.  Next we compute the image of $v_i$ for all $i$ and see that
 \begin{align*}
 \tau_{u_i'}(v_i) &= v_i + \dfrac{B\left( u_i + \dfrac{1}{a_i}Yv_i, v_i \right)}{\q\left(u_i + \dfrac{1}{a_i} Yv_i \right)}(u_i + \dfrac{1}{a_i} Yv_i) \\
 &= v_i + a_i \left( u_i + \dfrac{1}{a_i} Yv_i \right) \\
 &= v_i + a_iu_i + Yv_i.
 \end{align*}
 The map $Y'$ acts on $V$ by adding defective vectors to the $u_i$ and acting as the zero map on the $v_i$. So we have that $\q(Yw) = 0$ for all $w \in V$.  In the end we have that $(\tau_{\mathcal{U}'}, 0, \id) \in \Orth(\q,k)$ since $\tau_{\mathcal{U'}}$ is an orthogonal transvection involution.  The map $(\id, Y', \id) \in \Orth(\q,k)$, since $Y'$ can only add defective vectors to any element and so must preserve $\q$.  Finally  $(\id, 0, \rho) \in \Orth(\q,k)$ since it acts isometrically on the radical and trivially elsewhere.
\end{proof}

Now we can prove the following theorem.

\begin{thm} \label{geninv_character}
Two involutions $(\tau_{\mathcal{U},a},Y, \rho), (\tau_{\mathcal{X},b},Z, \gamma) \in \Orth(\q,k)$ are $\Orth(\q,k)$-conjugate if and only if there exists $(\varphi, X, \delta) \in \Orth(\q,k)$ such that
\begin{enumerate}
\item $\varphi \tau_{\mathcal{U},a} \varphi^{-1}= \tau_{\mathcal{X},b}$
\item $\delta \rho \delta^{-1} = \gamma$
\item $X\tau_{\mathcal{U},a} + \gamma X = Z \varphi + \delta Y$.
\end{enumerate}
\end{thm}
\begin{proof}
We can consider the elements of $\Orth(\q,k)$ as block diagonal matrices and compute
\begin{align*}
	\begin{bmatrix}
	\varphi & 0 \\
	X & \delta
	\end{bmatrix}
	\begin{bmatrix}
	\tau_{\mathcal{U},a} & 0 \\
	Y & \rho
	\end{bmatrix}
	\begin{bmatrix}
	\varphi & 0 \\
	X & \delta
	\end{bmatrix}^{-1} &=
	\begin{bmatrix}
	\varphi & 0 \\
	X & \delta
	\end{bmatrix}
	\begin{bmatrix}
	\tau_{\mathcal{U},a} & 0 \\
	Y & \rho
	\end{bmatrix}
	\begin{bmatrix}
	\varphi^{-1} & 0 \\
	\delta^{-1}X \varphi^{-1} & \delta^{-1}
	\end{bmatrix} \\
&=	\begin{bmatrix}
	\varphi \tau_{\mathcal{U},a} \varphi^{-1} & 0 \\
	(X\tau_{\mathcal{U},a} + \delta Y) \varphi^{-1} + \delta \rho \delta^{-1} X \varphi^{-1} & \delta \rho \delta^{-1}
	\end{bmatrix}	.
\end{align*}

The first two equations from the statement of the Proposition can be identified by setting the upper left and lower right diagonal equal to the corresponding block in $(\tau_{\mathcal{X},b}, Z, \gamma)$.  To get the final equation notice that the lower left block off the diagonal in the computation contains $\delta \rho \delta^{-1}$ which must be $\gamma$ by equation 2.  We then have the following equation
\[ (X\tau_{\mathcal{U},a} + \delta Y) \varphi^{-1} +\gamma X \varphi^{-1} = Z. \]
Multiplying $\varphi$ and then adding $\delta Y$ to both sides of the equation we arrive at
\[ X\tau_{\mathcal{U},a} +\gamma X  = Z \varphi + \delta Y. \]
\end{proof}

Notice that in Theorem \ref{geninv_character} property $1$ is equivalent to $(\mathcal{U},a)$ and $(\mathcal{X},b)$ being involution compatible, and property $2$ is equivalent to $\rho$ and $\gamma$ having equivalent quadratic signatures.

In general the existence of a triple $(\varphi, X, \delta)$ depends greatly on $\q$ and $k$.   We can consider the case when $\q$ is anisotropic when restricted to $\rad(V)$.  In this case if $(\tau_{\mathcal{U}}, Y, \rho)$ is an orthogonal involution then $\rho = \id$ and $Y=0$, since for any basis of $\rad(V)$ each basis vector will have a unique nonzero norm.  The other extreme would be if $\rad(V)$ is totally isotropic, so that every vector in $\rad(V)$ has norm zero.  In this case $\rho \in \textrm{GL}_s(k)$ where $s = \textrm{dim}(\rad(V))$ and $Y \in \textrm{Mat}_{r,s}(k)$, since there are no constraints contributed by $\q$ on $\rad(V)$ and adding vectors from the radical leaves $\q$ invariant on the image of any nonsingular subspace of $V$.

\vfill

\bibliographystyle{plain}

\end{document}